\documentclass[a4paper,12pt]{article}
\usepackage[dvipdfmx]{graphicx}
\usepackage{amsfonts}
\usepackage{amssymb,amsthm}
\usepackage{amsmath}
\usepackage{tikz}
\usepackage[ruled,vlined,lined,linesnumbered]{algorithm2e}
\usepackage{float}
\usepackage[font=small,labelfont=bf]{caption}

\usepackage[usenames,dvipsnames]{pstricks}

\newtheorem{theorem}{Theorem}[section]

\newtheorem{definition}[theorem]{Definition}
\newtheorem{remark}[theorem]{Remark}

\newtheorem{example}[theorem]{Example}

\begin{document}

\title{The multivariate bisection algorithm}

\author{Manuel L\' opez Galv\' an}

\author{Manuel L\'opez Galv\'an\footnote{Supported by Facultad de Ciencias Exactas y Naturales - Universidad de Buenos Aires - Argentina Intendente Güiraldes 2160 - Ciudad Universitaria - C1428EGA - Tel. (++54 +11)       			4576-3300 and IAM Instituto Argentino de Matem\'atica (CONICET) Saavedra 15 3rd floor - C1083ACA Buenos Aires Argentina  Tel.: 54 11 4954-6781.}}

\maketitle

\begin{abstract}
The aim of this paper is the study of the bisection method in $\mathbb{R}^n$. In this work we propose a multivariate bisection method supported by the Poincar\'e-Miranda theorem in order to solve non-linear system of equations. Given an initial cube verifying the hypothesis of Poincar\'e-Miranda theorem the algorithm performs congruent refinements throughout its center by generating a root approximation. Throughout preconditioning we will prove the local convergence of this new root finder methodology and moreover we will perform a numerical implementation for the two dimensional case. 

\end{abstract}

\section{Introduction}
The problem of finding numerical approximations to the roots of a non-linear system of equations was subject of various studies, different methodologies have been proposed between optimization and Newton's procedures. In \cite{Lehmer} D. H. Lehmer proposed a method for solving polynomial equations in the complex plane testing increasingly smaller disks for the presence or absence of roots. In other work, Herbert S. Wilf developed a global root finder of polynomials of one complex variable inside any rectangular region using Sturm sequences\cite{Wilf}.    

The classical Bolzano's theorem or Intermediate Value theorem ensure that a continuous function that changes sign in an interval has a root, that is, if $f:[a,b] \rightarrow \mathbb{R}$ is continuous and $f(a)f(b)<0$ then there exist $c \in (a,b)$ that $f(c)=0$. In the multidimensional case the generalization of this result is the known  Poincar\'e-Miranda theorem that ensures that if we have $f_1,...,f_n$ $n$-continuous functions of $n$ variables $x_1,...,x_n$ and the variables are subjected to vary between $a_i$ and $-a_i$ then if $f_i(x_1,..,a_i,..,x_n)f_i(x_1,..,-a_i,..,x_n)<0$ for all $x_i$ then there exist $c \in [-a_i, a_i]^{n}$ such that $f(c)=0$. This result was announced the first time by Poincar\'e in 1883 \cite{Poincare1} and published in 1884 \cite{Poincare2} with reference to a proof using homotopy invariance of the index. The result obtained by Poincar\'e has come to be known as the theorem of Miranda, who in 1940 showed that it is equivalent to the Brouwer fixed point \cite{Miranda}. For different proofs of the Poincar\'e-Miranda theorem in the $n$-dimensional case, see \cite{Kulpa}, \cite{Rouche}. 

\begin{theorem} \label{Poin-Mir_Teo} (Poincar\'e-Miranda theorem). Let $K$ be the cube 
$$K=\lbrace x\in \mathbb{R}^{n} : \vert x_j-\hat{x_j} \vert \leq \rho , \ j=1(1)n\rbrace$$
where $\rho \geq 0$ and $\mbox{\bf{F}}=(f_1,f_2,..,f_n): K \rightarrow \mathbb{R}^n$ a continuous map on $K$. Also let,
$$F_i^+=\lbrace x \in K: x_i=\hat{x_i} + \rho \rbrace, \  F_i^-=\lbrace x \in K: x_i=\hat{x_i} - \rho \rbrace$$
be the pairs of parallel opposite faces of the cube $K$. 

If for $i= 1 (1)n$ the $i$-th component $f_i$ of $\mbox{\bf{F}}$ has opposite sign or vanishes on the corresponding opposite faces $F_i^+$ and $F_i^-$ of the cube $K$, i.e.     
\begin{eqnarray} 
f_i(x)f_i(y)\leq 0, \ x\in F_i^+, y\in  F_i^- \label{PM}
\end{eqnarray} 
then the mapping $\mbox{\bf{F}}$ has at least one zero point $r=(r_1,r_2,..,r_n)$ in K.      
\end{theorem}

Throughout this paper we will recall the opposite signs condition \ref{PM} as the Poincar\'e-Miranda property P.M.. The aim of this work is to develop a bisection method that allows us solve non-linear system of equations $\mbox{\bf{F}}(X)=0, \ X=(x_1,x_2,..,x_n)$ using the above Poincar\'e-Miranda theorem. The idea of the algorithm will be  similar as in the classical one dimensional algorithm, we perform refinements of the cube domain in order to check the sign conditions on the parallel faces. In one dimension it is clear that an initial sign change in the border of an interval produces other sign change in a half partition of it but in several dimension we cannot guarantee that the Poincar\'e-Miranda conditions maintain after a refinement. Even if $r$ is an exact solution, there may not be any such $K$ (for which \ref{PM} holds). However, J. B. Kioustelidis \cite{KIOUSTELIDIS} has pointed out that, for $\hat{x}$ close to a simple solution (where the Jacobian is nonsingular) of $\mbox{\bf{F}}(X)$, Miranda's theorem will be applicable to some equivalent system for suitable $K$. Therefore in case of a fail in the sign conditions with the original system, we should try to transform it. The idea will be find an equivalent system throughout non-linear preconditioning where the equations are better balanced in the sense that the new system could be close to some hyperplane in order to improve the chances to check the sign conditions in some member of the refinement.

We will denote the infinite norm by $\Vert x \Vert_{\infty}=\max \lbrace\vert x_1\vert,..,\vert x_n\vert \rbrace$, the Euclidean norm by $\Vert x \Vert_2=\sqrt{x_1^2+..+x_n^2}$ and the $1$-norm by $\Vert x \Vert_1=\vert x_1 \vert +..+\vert x_n \vert$. Given a vector norm on $\mathbb{R}^{n}$, the associated matrix norm for a matrix $M \in \mathbb{R}^{n\times n}$ is defined by $$\Vert M \Vert_p = \max_{\Vert x \Vert_p = 1} \Vert Mx \Vert_p  \ \mbox{where}  \ p=\infty, 2 \ \mbox{or} \ 1$$   
It is know that in the case of the $\infty$-matrix norm it can be expressed as a maximum sum of its row, that is if $M=(m_{ij})$ then $\Vert M \Vert_{\infty} = \max_{i=1,..,n} \sum_{j=1}^{n} \vert m_{ij}\vert$, therefore it is easy to see that a sequence of matrices $(M_k)_k$ converge if and only if their coordinates converge. 
Since the domains involved are multidimensional cubes, the most proper norm to handle the distance will be the  $\infty$-norm.

We will accept $r$ as a root with a small tolerance level $\delta$ if $\Vert \mbox{\bf{F}}(r) \Vert_p \leq \delta$.

\section{The algorithm and its description}
This section gives a step-by-step description of our algorithm, the core of it meets in the classical bisection algorithm in one dimension. 
\begin{definition} A $2^n$-refinement of a cube $K\subset \mathbb{R}^n$ is a refinement into $2^n$ congruent cubes $Q= \lbrace K^1, K^2,..,K^{2^n}  \rbrace$.



\begin{figure}[H]  
\begin{center}
\includegraphics[height=1.1in,width=1.1in]{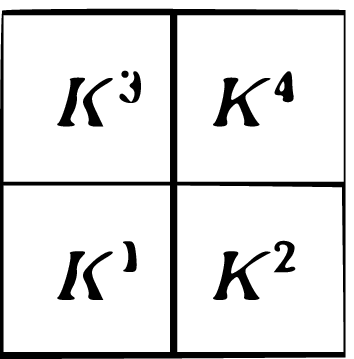}\\ 
\caption{4-refinement in $\mathbb{R}^2$.}
\includegraphics[height=1.2in,width=1.2in]{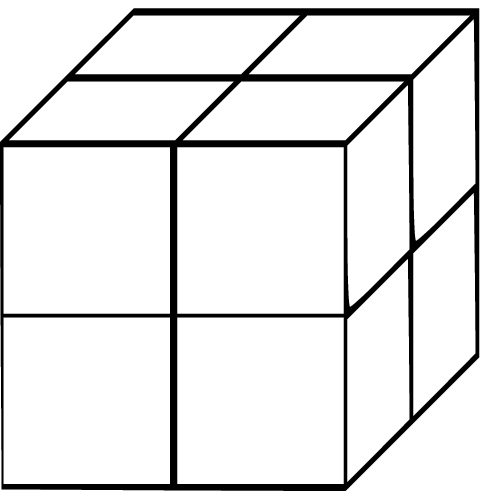}\\ 
\caption{8-refinement in $\mathbb{R}^3$.}
\end{center}
\end{figure}

We say that a $2^n$-refinement of $Q$ verifies the Poincar\'e-Miranda condition if there exist $K^l \in Q$ such that $\mbox{\bf{F}} : K^l \rightarrow \mathbb{R}^n$ verifies the condition of Theorem \ref{Poin-Mir_Teo}.    
\end{definition}

Given a system  $\mbox{\bf{F}}(X)=0$, the preconditioned system is $\mbox{\bf{G}}(X)=M\mbox{\bf{F}}(X)=0$ for some matrix $M$ such that the jacobian at $X_0$ verifies $D\mbox{\bf{G}}(X_0)=\mbox{Id}$. Since $D\mbox{\bf{G}}(X_0)=MD\mbox{\bf{F}}(X_0)$ it turns out that $M=D\mbox{\bf{F}}(X_0)^{-1}$ and it is clear that the preconditioned system is an equivalent system of $\mbox{\bf{F}}$ and both have the same roots. After preconditioning, the equations in $\mbox{\bf{G}}(X)=0$ are close to a hyperplane having equation $x=k_1$ and $y=k_2$, where $k_i$ is some constant. This fact comes from the Taylor expansion of $\mbox{\bf{G}}$ around $X_0$, indeed if $X$ is close to $X_0$ then  
$$\mbox{\bf{G}}(X)\approx \mbox{\bf{G}}(X_0) + D\mbox{\bf{G}}(X_0)(X-X_0) = \mbox{\bf{G}}(X_0) + X-X_0$$         
and therefore it is clear that the equations are close to some hyperplane. Moreover if $X_0$ is nearly a zero point of $\mbox{\bf{F}}$ then, $$\mbox{\bf{G}}(X)\approx X-X_0$$ and therefore it will behave like the components of $X - X_0$, and take nearly opposite values on the corresponding opposite faces of the cube.

{\bf \subsection{Algorithm procedure}}
The multivariate bisection algorithm proceeds as follows:
\begin{enumerate}
\item We start choosing  an initial guess $K_0=[a_{1_0},b_{1_0}]\times [a_{2_0},b_{2_0}]\times ...\times [a_{n_0},b_{n_0}] \subset \mathbb{R}^n$ verifying the Poincar\'e-Miranda condition on $\mbox{\bf{F}}$.
\item We locate the center $$c_1=\bigg( \frac{a_{1_0}+ b_{1_0}}{2}, \frac{a_{2_0}+ b_{2_0}}{2}, ..., \frac{a_{n_0}+ b_{n_0}}{2} \bigg)$$ of $K_0$. 
\item Generate a first $2^n$-refinement $Q_1$ through $c_1$.   
\item If $Q_1$ verifies the Poincar\'e-Miranda condition, let $K_1=[a_{1_1},b_{1_1}]\times [a_{2_1},b_{2_1}]\times ...\times [a_{n_1},b_{n_1}] $ be the quarter of $Q_1$ where the conditions of Theorem \ref{Poin-Mir_Teo} are verified, we chose $$c_2=\bigg(  \frac{a_{1_1}+ b_{1_1}}{2}, \frac{a_{2_1}+ b_{2_1}}{2}, ..., \frac{a_{n_1}+ b_{n_1}}{2}   \bigg)$$ the center of $K_1$. 
If $Q_1$ does not verify the Poincar\'e-Miranda condition we preconditioning the system in $c_1$ setting $$\mbox{\bf{G}}_1(X):=D\mbox{\bf{F}}(c_1)^{-1}\mbox{\bf{F}}(X)$$  
and then we check again the sign conditions with the preconditioned system $\mbox{\bf{G}}_1(X)$ in $Q_1$.
 
This recursion is repeated while the Poincar\'e-Miranda condition are verified, generating a sequence of equivalent system \[
\mbox{\bf{G}}_k(X):= \left\{ \begin{array}{lcl}
\mbox{\bf{G}}_{k-1}(X) \ \mbox{ if} \ \mbox{\bf{G}}_{k-1} \ \mbox{verifies P.M. in} \ Q_k \\
& & \\
D\mbox{\bf{G}}_{k-1}(c_{k})^{-1}\mbox{\bf{G}}_{k-1}(X) \ \mbox{if} \ \mbox{\bf{G}}_{k-1} \ \mbox{does not verify P.M. in} \ Q_k
\end{array}
\right.
\] and a decreasing cube sequence $K_k$, such that
\begin{eqnarray} 
 & K_{k+1} \subset K_k= [a_{1_k},b_{1_k}]\times ...\times [a_{n_k},b_{n_k}] \ \ \mbox{where the vertices verify }  \nonumber\\
 & a_{j_0}\leq   a_{j_1} \leq  a_{j_2}  \leq  \ldots   \leq  a_{j_k} \leq \ldots  \leq   b_{j_0} \label{x1} \\
 & b_{j_0}\geq  b_{j_1} \geq  b_{j_2} \geq  \ldots  \geq b_{j_k} \geq \ldots \geq  a_{j_0} \label{x2}
 \end{eqnarray}
 for each $j=1(1)n$ and where the length of the current interval $[a_{j_k},b_{j_k}]$ is a half of the last iteration, 
\begin{eqnarray} 
& a_{j_k}-b_{j_k}  = \dfrac{a_{j_{k-1}}-b_{j_{k-1}}}{2} = \ldots = \dfrac{a_{j_0}-b_{j_0}}{2^k} \label{des_x} 
\end{eqnarray}
The root's approximation after $k$-th iteration will be, $$c_k=\bigg( \frac{a_{1_k}+ b_{1_k}}{2}, \frac{a_{2_k}+ b_{2_k}}{2}, ..., \frac{a_{n_k}+ b_{n_k}}{2} \bigg)$$
and the method is stopped until the zero's estimates gives sufficiently accuracy or until the Poincar\'e-Miranda condition leaves to maintain. 
\end{enumerate}

\begin{remark} It is easy to see that the $k$th-preconditioning system $\mbox{\bf{G}}_{k}(X)=(g_{1_k}(X),..,g_{n_k}(X))$ can be expressed as $D\mbox{\bf{F}}(c_{k})^{-1}\mbox{\bf{F}}(X)$, indeed, by induction suppose that it is true for $k-1$, then differencing and valuing in $c_k$ we have, $$\mbox{\bf{G}}_{k}(X)=D\mbox{\bf{G}}_{k-1}(c_{k})^{-1}\mbox{\bf{G}}_{k-1}(X)$$$$=D\mbox{\bf{G}}_{k-1}(c_{k})^{-1}D\mbox{\bf{F}}(c_{k-1})^{-1}\mbox{\bf{F}}(X)= D\mbox{\bf{F}}(c_k)^{-1} D\mbox{\bf{F}}(c_{k-1}) D\mbox{\bf{F}}(c_{k-1})^{-1}\mbox{\bf{F}}(X)$$
$$= D\mbox{\bf{F}}(c_k)^{-1}\mbox{\bf{F}}(X).$$
  
\end{remark}

Since we cannot always ensure that a refinement of a given cube will verify the Poincar\'e-Miranda condition, we cannot ensure the converge for any map that only has a sign change in a given initial cube. So, in case of a fail in the sign conditions in some step, we try to rebalance the system using preconditioning in the center of the current box recursion. The preconditioning allows us to increase the chances to be more often in the sign conditions and therefore keep going with the quadrisection procedure in order to get a better root's approximation. In \cite{KIOUSTELIDIS}, J. B. Kioustelidis found sufficient conditions for the validity of the Poincar\'e-Miranda Miranda condition for preconditioning system, there it was proved that the sign conditions are always valid if the center of the cube $K$ is close enough to some root of $\mbox{\bf{F}}$. So, if we start the multivariate bisection algorithm with an initial guess close to some root, Kioustelidis's theorem will guarantee the validity of Poincar\'e-Miranda in each step of our method allowing the local convergence of it.    

In the next theorem we will prove the local convergence for the multivariate bisection algorithm when we preconditioning in each step.

\begin{theorem}\label{conv_teo} Let $\mbox{\bf{F}}=(f_1,...,f_n) : K_0 \rightarrow \mathbb{R}^n$ be a $C^2$ map defined on the cube $K_0=\lbrace x\in \mathbb{R}^{n} : \Vert x-c_1 \Vert_{\infty} \leq \rho \rbrace=[a_{1_0},b_{1_0}]\times..\times [a_{n_0},b_{n_0}]$ with $\rho$ small enough verifying the Poincar\'e-Miranda sign condition; assume that $D\mbox{\bf{F}}(X)$ is invertible for all $X \in K_0$, furthermore suppose that we perform the preconditioning in each step then the multivariate bisection algorithm generates a sequence $c_k$ such that
\begin{enumerate}
\item Starting at $K_0$, $c_k \stackrel{\Vert .\Vert}\longrightarrow r$ with $\mbox{\bf{F}}(r)=0$.
\item $\Vert c_k - r\Vert_2 \leq \dfrac{\sum^n_{j=1} b_{j_0} - a_{j_0}  }{2^{k}}$.    
\end{enumerate}
\end{theorem}

\begin{proof}\leavevmode 
\begin{enumerate}
\item The Poincar\'e-Miranda sign conditions guarantee the existence of a root inside $K_0$ and given a refinement $Q_1$ of $K_0$ since $\rho$ is small enough Item c of Theorem 2 in \cite{KIOUSTELIDIS} guarantees the validity of Poincar\'e-Miranda sign conditions for a member of $Q_1$. Performing successive refinements we will always find a member $K_k$ of the refinement $Q_k$ verifying the sign conditions for the preconditioned system $\mbox{\bf{G}}_{k}(X)$. 
For each $j=1(1)n$ the sequences $(a_{j_k})_k, (b_{j_k})_k$ are monotones and bounded and therefore they converge. From equation \ref{des_x} we have for each $j=1(1)n$,  
\begin{eqnarray} 
&   \displaystyle{\lim_{k \rightarrow \infty}} a_{j_k} =\lim_{k \rightarrow \infty} b_{j_k} = r_j  \label{xlim} 
\end{eqnarray}
       
and from the border conditions, 
\begin{eqnarray}
&   g_{1_k}(a_{1_k},x_2,..,x_n)g_{1_k}(b_{1_k},x_2,..,x_n)   \leq 0  \ \ \forall \  x \in  \widehat{K_k}^1 \label{desf_1} \\ 
&   ........  \nonumber \\
&   g_{j_k}(x_1,..,x_{j-1},a_{j_k},x_{j+1},..,x_n)g_{j_k}(x_1,..,x_{j-1},b_{j_k},x_{j+1},..,x_n)   \leq 0  \  \forall \  x \in  \widehat{K_k}^j   \nonumber \\
&   ........\nonumber \\
&   g_{n_k}(x_1,..,x_{n-1},a_{n_k})g_{n_k}(x_1,..,x_{n-1},b_{n_k})   \leq 0  \ \ \forall \ \   x \in  \widehat{K_k}^n   \nonumber
\end{eqnarray}

where $\widehat{K_k}^j$ means that the $j$ coordinate of $K_k$ is omitted.
Since the diameter of $K_k$ tends to zero by Cantor's intersection theorem the intersection of the $K_k$ contains exactly one point, $$\lbrace p \rbrace=\bigcap_{n=0}^{\infty} K_k$$
and the equations \ref{xlim} guarantee that $p=(r_1,..,r_n)$. Then, we can evaluate equations \ref{desf_1} in $p=(r_1,..,r_n)$ getting, 

\begin{eqnarray}\label{desf_2}
& g_{1_k}(a_{1_k},r_2,..,r_n)g_{1_k}(b_{1_k},r_2,..,r_n)   \leq 0  \ \forall \ k \in \mathbb{N} \\ 
& ........ \nonumber \\
& g_{n_k}(r_1,..,r_{n-1},a_{n_k})g_{n_k}(r_1,..,r_{n-1},b_{n_k})   \leq 0  \ \forall \ k  \in \mathbb{N} \nonumber
\end{eqnarray}

It is clear that, 
$$c_k=\bigg( \frac{a_{1_k}+ b_{1_k}}{2}, \frac{a_{2_k}+ b_{2_k}}{2}, ..., \frac{a_{n_k}+ b_{n_k}}{2} \bigg) \underset{k\rightarrow \infty}\longrightarrow (r_1,..,r_n)=r$$
then by the continuity of $D\mbox{\bf{F}}$ and the continuity of the inversion in the $\infty$-matrix norm we have $$D\mbox{\bf{F}}(c_k)^{-1}\rightarrow D\mbox{\bf{F}}(r)^{-1}$$
Let $G(X)=D\mbox{\bf{F}}(r)^{-1}F(X)=(g_1(X),..,g_n(X))$, since $$\Vert D\mbox{\bf{F}}(c_k)^{-1}F(X)- D\mbox{\bf{F}}(r)^{-1}F(X)\Vert_{\infty} \leq \Vert D\mbox{\bf{F}}(c_k)^{-1}- D\mbox{\bf{F}}(r)^{-1}\Vert_{\infty} \Vert F(X) \Vert_{\infty} \rightarrow 0$$ for each $X \in K_0$ then we get the punctual convergence  for each coordinate function $$g_{j_k}(X)\underset{k\rightarrow \infty}\rightarrow g_j(X).$$ 
From equations \ref{xlim} we have, $$g_{j_k}(r_1,..,r_{j-1},a_{j_k},r_{j+1},..,r_n) \rightarrow g_j(r_1,..,r_j,..,r_n) \ \mbox{for each} \ j=1(1)n$$
$$g_{j_k}(r_1,..,r_{j-1},b_{j_k},r_{j+1},..,r_n) \rightarrow g_j(r_1,..,r_j,..,r_n)  \ \mbox{for each} \ j=1(1)n,$$ therefore taking limit in equations \ref{desf_2} we get
$$g_{j}(r_1,r_2,..,r_n)^2   \leq 0 \ \forall j=1(1)n  $$ and finally it is clear that $\mbox{\bf{F}}(r_1,..,r_n)=0$.

\item
Let $(c_{j_k})_k \ j=1(1)n$, be the coordinates of the sequence $(c_k)_k$ we have following estimation 
\begin{eqnarray}
& \vert c_{j_k} -r_j\vert \leq \dfrac{b_{j_0} - a_{j_0}}{2^k} \ \mbox{for each} \ j=1(1)n. 
\end{eqnarray}

Indeed, since the sequences $(a_{j_k})$ and $(b_{j_k})$ are monotones and bounded by $r_j$ we get for each $j=1(1)n$,
\begin{eqnarray*}
& c_{j_k} - r_j=\dfrac{a_{j_{k-1}}}{2} +\dfrac{b_{j_{k-1}}}{2} - r_j &\leq  \dfrac{a_{j_{k-1}}}{2} +\dfrac{b_{j_{k-1}}}{2} - a_{j_{k-1}} \\
&  & =\dfrac{b_{j_{k-1}}}{2}  - \dfrac{a_{j_{k-1}}}{2} =\dfrac{b_{j_0}-a_{j_0}}{2^k} 
\end{eqnarray*}
On the other hand, 
\begin{eqnarray*}
& c_{j_k} - r_j=\dfrac{a_{j_{k-1}}}{2} +\dfrac{b_{j_{k-1}}}{2} - r_j &\geq  \dfrac{a_{j_{k-1}}}{2} +\dfrac{b_{j_{k-1}}}{2} - b_{j_{k-1}} \\
&  & =-\big(\dfrac{b_{j_{k-1}}}{2} - \dfrac{a_{j_{k-1}}}{2}\big) =-\big(\dfrac{b_{j_0}-a_{j_0}}{2^k} \big)
\end{eqnarray*}
Therefore,
\begin{eqnarray*}
& \Vert c_k - r\Vert_2 = \sqrt{\sum^{n}_{j=1}(c_{j_k} -r_j)^2} \leq \sum^{n}_{j=1} \vert c_{j_k} -r_j\vert  \\
& \leq \dfrac{\sum^{n}_{j=1} b_{j_0} - a_{j_0}}{2^k}
\end{eqnarray*}  
\end{enumerate}
\end{proof}

As in the classical one dimensional bisection algorithm Item 2 of theorem \ref{conv_teo} gives a way to determine the number of iterations that the bisection method would need to converge to a root to within a certain tolerance. The number of iterations needed, $k$, to achieve the given tolerance $\delta$ is given by,
$$ k=\log_2 \bigg( \frac{\sum^n_{j=1} b_{j_0} - a_{j_0} }{\delta} \bigg)=\dfrac{\log \big( \sum^n_{j=1} b_{j_0} - a_{j_0} \big ) - \log \delta }{\log 2}$$   

The following example shows an infinite application of the bisection algorithm in $\mathbb{R}^2$ with non-preconditioning. 
\begin{example} \label{conv_example} Consider the map $\mbox{\bf{F}}(x,y)=(y+x-1, y-e^{-x^2})$, we start checking the Poincar\'e-Miranda condition on $K_0=[0,1]\times [0,1]$,      
\begin{eqnarray*}
& f_1(0,y)=y-1 \leq 0, \  \  f_1(1,y)= y \geq 0 \\
& f_2(x,0)=-e^{-x^2} < 0, \ \  f_2(x,1)=1-e^{-x^2}\geq 0 
\end{eqnarray*}
then if we considerate the $K^3$ quarter of each quadrisection, we can check that it always will verify the Poincar\'e-Miranda condition and therefore it is not necessary preconditioning in each step. Let $a_{1_{k}}=0$, $b_{1_{k}}=\frac{1}{2^k}$, $a_{2_{k}}=1-\frac{1}{2^k}$ and $b_{2_{k}}=1$ be the coordinates of the $k$-th $K^3$ quarter rectangle, we have to prove that $$f_1(a_{1_{k}},y)f_1(b_{1_{k}},y)=(y-1)(\frac{1}{2^k}+y-1)\leq 0, \   1-\frac{1}{2^k}\leq y\leq 1.$$
$$f_2(x,a_{2_{k}})f_2(x,b_{2_{k}})=(1-\frac{1}{2^k}-e^{-x^2})(1-e^{-x^2})\leq 0, \   0\leq x\leq\frac{1}{2^k}$$
Indeed, the first inequality is clear and follows directly from the domain of $y$, the second follows from the fact that the domain of $x$ implies that  $$1-\frac{1}{2^k}\leq 1-x \leq e^{-x^2} $$ getting $$1-\frac{1}{2^k}-e^{-x^2}\leq 0 ,  \  k \in \mathbb{N} \ \mbox{and} \ 1-e^{-x^2}\geq 0$$
 The $k$-th root's approximation is, $$c_k= \bigg(\dfrac{\dfrac{1}{2^{k-1}}}{2}, \dfrac{1-\dfrac{1}{2^{k-1}}+1}{2} \bigg)=\bigg( \dfrac{1}{2^k}, 1-\dfrac{1}{2^k}\bigg) \displaystyle{\longrightarrow_{k\rightarrow \infty}} (0,1)$$ 
and the error verifies $$\Vert c_k - (0,1)\Vert =  \sqrt{\bigg(\dfrac{1}{2^k}\bigg)^2 + \bigg(\dfrac{1}{2^k}\bigg)^2}= \dfrac{\sqrt{2}}{2^k}\leq \dfrac{2}{2^{k}}$$
\end{example}

\section{Implementation, performance and testing}
Throughout this section we will focus in the implementation and performance of the bisection algorithm in $\mathbb{R}^2$. The bisection algorithm was developed in \textsc{Matlab} in a set of functions running from a main function. In order to check the P.M. conditions for the function $\mbox{F}=(f_1,f_2)$ we need to compute the intervals $f_i(F_i^+), f_i(F_i^-)$ ($i=1,2$) and one way to achieve this is by using Interval Analysis (IA). 
IA was marked by the appearance of the book Interval Analysis by Ramon E. Moore in 1966 \cite{Moore} 
and it gives a fast way to find an enclosure for the range of the functions. A disadvantage of IA is the well known overestimation. If intervals $f_i(F_i^+), f_i(F_i^-)$ are available then the P.M. follows from the condition 
\begin{eqnarray}
& \sup \lbrace y : y \in f_i(F_i^-) \rbrace \leq 0 \leq \inf  \lbrace y : y \in f_i(F_i^+) \rbrace \label{PMintform1} \\     
& or \nonumber \\ 
& \sup \lbrace y : y \in f_i(F_i^+) \rbrace \leq 0 \leq \inf  \lbrace y : y \in f_i(F_i^-) \rbrace \label{PMintform2}
\end{eqnarray}
Interval-Valued Extensions of Real Functions gives a way to find an enclosure of the range of a given real-valued function. Most generally, if we note by $[\mathbb{R}]$ the set of all finite intervals, we say that $ [f] :[\mathbb{R}]^n \rightarrow [\mathbb{R}]$ is an interval extension of $f:\mathbb{R}^n \rightarrow \mathbb{R}$ if $$[f](X)\supseteq \lbrace f(x): x\in X\rbrace$$     
where $X=(X_1,...,X_n)$ represents a vector of intervals. There are different kind of interval functional extensions; if we have the formula of a real-valued function $f$ then the natural interval extension is achieved by replacing the real variable $x$ with an interval variable $X$ and the real arithmetic operations with the corresponding interval operations. Another useful interval extension is the mean value form. Let $m=m(X)$ be the center of the interval vector $X$ and let $[\frac{\partial f_i}{\partial x_i}]$ be an interval extension of $\frac{\partial f_i}{\partial x_i}$ by the mean value theorem we have $$f(X)\subseteq [f_{mv}](X)= f(m)+ \sum^n_{i=1} [\frac{\partial f_i}{\partial x_i}](X)(X_i-m_i)$$
$[f_{mv}](X)$ is the mean value extension of $f$.

Let $[f_i](F_i^+), [f_i](F_i^-)$ be an interval extension of $f_i(F_i^+), f_i(F_i^-)$, then it is clear that if 
\begin{eqnarray}
& \sup \lbrace y : y \in [f_i](F_i^-) \rbrace \leq 0 \leq \inf  \lbrace y : y \in [f_i](F_i^+) \rbrace \label{PMintform3} \\     
& or \nonumber \\ 
& \sup \lbrace y : y \in [f_i](F_i^+) \rbrace \leq 0 \leq \inf  \lbrace y : y \in [f_i](F_i^-) \rbrace \label{PMintform4}
\end{eqnarray}
equations \ref{PMintform1} and \ref{PMintform2} are also true. So, in order to check the P.M. conditions along the edges we will compute equations \ref{PMintform3} and \ref{PMintform4}.        
 
Various interval-based software packages for \textsc{Matlab} are available, we have chosen the well known \textsc{INTLAB} toolbox \cite{INTLAB}. The toolbox has several interval class constructor for intervals, affine arithmetic, gradients, hessians, slopes and more. 
Ordinary interval arithmetic has sometimes problems with dependencies and wrapping effect given large enclosures of the range and therefore overestimating the sign behaviour. A way to fight with this is affine arithmetic. In affine arithmetic an interval is stored as a midpoint $X_0$ together with error terms $E_1,...,E_k$ and it represents $$X= X_0+ E_1U_1+E_2U_2+...+E_kU_k$$ where $U_1,...,U_k$ are parameters independently varying within $[-1,1]$. 
In case of get wrong signs for $f_i(F_i^-)$ and $f_i(F_i^+)$ and dismiss the possibility of a not very sharp estimation of IA we also compute the interval extension but now using the affine arithmetic.

Other way to improve the enclosure of the range and get sharper lower and upper bounds is throughout subdivision or refinements. In this methodology we perform subdivision of the domain and then we take the union of interval extensions over the elements of the subdivision; this procedure is called a refinement of $[f]$ over $X$. Let $N$ be a positive integer we define 
\begin{equation}
X_{i,j}=[\inf X_i +(j-1)\frac{w(X_i)}{N},\inf X_i + j\frac{w(X_i)}{N}] \ \ j=1,..,N \nonumber
\end{equation}
We have $X_i=\cup^N_{j=1} X_{i,j}$ and $w(X_{i,j})=\dfrac{w(X_i)}{N}$ and furthermore, 
\begin{equation}
X=\cup^N_{j_i=1}(X_{1,j_1},...,X_{n,j_n}) \ \mbox{with} \ w(X_{1,j_1},...,X_{n,j_n})=\dfrac{w(X)}{N} \nonumber. 
\end{equation}
The interval quantity $$[f]_N(X)= \cup^N_{j_i=1}[f](X_{1,j_1},...,X_{n,j_n})$$ is the refinement of $[f]$ over $X$.
   
The algorithms that we have performed to compute equations \ref{PMintform3} and \ref{PMintform4} combines all the above methodologies and were adapted from \cite{Moore2}. In the following steps we summarize the routines that we have performed.  
The mean value extension was implemented using an approximation of $[\frac{\partial f_i}{\partial x_i}]$ throughout the central finite difference of the natural interval extension of $f$, that is
\begin{equation}
[\frac{\partial f_i}{\partial x_i}](X)\approx \dfrac{[f](X+0.0001)-[f](X-0.0001)}{2 \ 0.0001}\nonumber. 
\end{equation}       
Algorithm \ref{mean_value_extension} shows the routine for the mean value extension.

\begin{algorithm}[H]\label{mean_value_extension}
\small
\caption{Function {\it meanValue}, Computes the mean value extension}
\KwData {$f,X$}
\KwResult {returns $Fmv$ the value for the mean value extension form for $f$ evaluated over the interval $X$.}
\mbox{m} $\leftarrow$ mid(X)\\ 
\mbox{fm} $\leftarrow$ $f(m)$\\ 
\mbox{derf} $\leftarrow$ $f(X+0.0001)-f(X-0.0001)/(2 * 0.0001)$\\ 
\mbox{Fmv} $\leftarrow$ fm + derf * (X-m)\\
\end{algorithm}

The refinement procedure was implemented twice, one for the case of mean value extension and other for the affine arithmetic implementation. Algorithm \ref{mean_value_refinement} computes the mean value extension over an uniform refinement of the interval $X$ with $N$ subintervals and Algorithm \ref{affine_interval_refinement} computes the natural extension using affine arithmetic.

\begin{algorithm}[H]\label{mean_value_refinement}
\small
\caption{Function {\it meanValueRefinement}, Computes the refinement procedure using mean value extension}
\KwData {$f,X,N$}
\KwResult {returns $Y$ the value for the mean value extension form for $f$ evaluated over a partition of $X$.}

\mbox{h} $\leftarrow$ $(\sup(X)-\inf(X))/N$\\
\mbox{xi} $\leftarrow$ $\inf(X)$\\
\mbox{x1} $\leftarrow$ xi\\
\For{$i=1:N$}{ xip1  $\leftarrow$ x1 + $i$*h \\
               Xs($i$) $\leftarrow$ infsup(xi,xip1) \  {\it $\setminus\setminus$ Interval class constructor for each subinterval.}  \\ 
               xi $\leftarrow$ xip1   }

Xs(N) $\leftarrow$ infsup($\inf(Xs(N)),\sup(X))$\\
\mbox{Y} $\leftarrow$  {\it meanValue}($f$,Xs(1))\\
\If {$N>1$} {\For{$i=2:N$}{\mbox{Y} $\leftarrow$ hull(Y,{\it meanValue}($f$,Xs($i$)) {\it $\setminus\setminus$ take the union of mean extension. } }  }
\end{algorithm}

\begin{algorithm}[H]\label{affine_interval_refinement}
\small
\caption{Function {\it affineIntervalRefinement}, Computes the refinement procedure using affine natural extension}
\KwData {$f,X,N$}
\KwResult {returns $Y$ the value for the affine natural extension form for $f$ evaluated over a partition of $X$.}

\mbox{h} $\leftarrow$ $(\sup(X)-\inf(X))/N$\\
\mbox{xi} $\leftarrow$ $\inf(X)$\\
\mbox{x1} $\leftarrow$ xi\\
\For{$i=1:N$}{ xip1  $\leftarrow$ x1 + $i$*h \\
               Xs($i$) $\leftarrow$ infsup(xi,xip1) \  {\it $\setminus\setminus$ Interval class constructor for each subinterval.}  \\ 
               xi $\leftarrow$ xip1   }

Xs(N) $\leftarrow$ infsup($\inf(Xs(N)),\sup(X))$\\
\mbox{Y} $\leftarrow$  $f$(affine(Xs(1)))\\
\If {$N>1$} {\For{$i=2:N$}{\mbox{Y} $\leftarrow$ hull(Y,$f$(affine(Xs($i$))) {\it $\setminus\setminus$ take the union of natural affine extension. } }  }
\end{algorithm}
	
Now we are ready to compute equations \ref{PMintform3} and \ref{PMintform4} using the above algorithms. Let $K^l=[a_{1_l},b_{1_l}]\times [a_{2_l},b_{2_l}]=I_{1_l}\times I_{2_l}$ be a member of the refinement $Q$ and let
\begin{equation}
 f_{11}=f_2(\cdot,a_{2_l}): [a_{1_l},b_{1_l}] \rightarrow \mathbb{R}, \  f_{12}=f_2(\cdot,b_{2_l}): [a_{1_l},b_{1_l}] \rightarrow \mathbb{R} \nonumber
\end{equation} 
\begin{equation}
  f_{21}=f_1(a_{1_l},\cdot): [a_{2_l},b_{2_l}]\rightarrow \mathbb{R}, \ f_{22}=f_1(b_{1_l},\cdot): [a_{2_l},b_{2_l}]\rightarrow \mathbb{R} \nonumber
\end{equation} 

be the coordinate functions on the edges of $K^l$ ($l=1...4$), Algorithm \ref{Compute_Sign_K^l} summarizes the routine that we have performed using IA in order to compute the sign along the edges.      

\bigskip

\begin{algorithm}[H]\label{Compute_Sign_K^l}
\small
\caption{Function {\it posneg}, Computes the sign along the edges of $K^l$}
\KwData{$f_{ij}$ ,$I_{il}=[a_{i_l},b_{i_l}]$, N }
\KwResult{returns sign$f_{ij}$, the sign of $f_{ij}$ on $I_{il}$, 1 means positive, -1 negative and NaN indicates an empty output when the sign is not constant.}
\mbox{Dom} $\leftarrow$ infsup($a_{i_l},b_{i_l}$) \ \ {\it $\setminus\setminus$ interval class constructor} \\
\mbox{Fmv} $\leftarrow$ {\it meanValueRefinement}($f_{ij}$,Dom,N) \ \ {\it $\setminus\setminus$ computes the mean interval extension}  \\
\mbox{extmin} $\leftarrow$ $\inf$(Fmv) \\
\mbox{extmax} $\leftarrow$ $\sup$(Fmv) \ \ {\it $\setminus\setminus$ computes the max and min of the mean extension}\\
\If {\mbox{\rm{extmin}} $\geq 0$} { 
    sign$f_{ij}$ $\leftarrow$ 1 \\
    \mbox{return} }

\If {\mbox{\rm{extmax}} $\leq 0$} { 
    sign$f_{ij}$ $\leftarrow$ -1 \\
    \mbox{return} }
    
\mbox{aff} $\leftarrow$ {\it affineIntervalRefinement}($f_{ij}$,Dom,N) \ \  {\it $\setminus\setminus$ computes the affine interval extension}\\
\mbox{extmin} $\leftarrow$ $\inf$(Fmv) \\
\mbox{extmax} $\leftarrow$ $\sup$(Fmv) \ \ {\it $\setminus\setminus$ computes the max and min of the affine extension}\\
\If {\mbox{\rm{extmin}} $\geq 0$} { 
    sign$f_{ij}$ $\leftarrow$ 1 \\
    \mbox{return} }

\If {\mbox{\rm{extmax}} $\leq 0$} { 
    sign$f_{ij}$ $\leftarrow$ -1 \\
    \mbox{return} }

sign$f_{ij}$ $\leftarrow$ NaN    
\end{algorithm}
  
Algorithm \ref{bisection Algorithm} summarizes the implementation of the Bisection Algorithm that we have performed in \textsc{Matlab}. 

\begin{algorithm}[H]\label{bisection Algorithm}
\footnotesize
\caption{Bisection  algorithm}
\KwData{$K_0$, F$=(f_1,f_2)$ system to solve, DF Jacobian of F, $\delta$, N}
\KwResult{$c$ root's approximation}
\eIf{$K_0$ \mbox{\rm verifies P.M. } } {$c \leftarrow $ center of $K_0$;\\
   error $\leftarrow \Vert \mbox{F}(c) \Vert$;\\ 
   stop $\leftarrow$ 1;\\   
   F1orig $\leftarrow$ $f_1$;\\
   F2orig $\leftarrow$ $f_2$;\\   
   \While{ \mbox{\rm (error} $\ >$  $\delta$ \mbox{\rm )}  $\wedge$ \mbox{\rm (stop} $\ < 3 )$ } {        ($K^1,K^2,K^3,K^4$) $\leftarrow$ Generate a refinement of $K_0$ throughout $c$;\\
(sign$f_{11}$,sign$f_{12}$,sign$f_{21}$,sign$f_{22}$) $\leftarrow$ {\it posneg} ( $f_{ij},I_{i1}$,N) \ $i,j=1,2$ {\it $\setminus\setminus$ Here we use Algorithm \ref{Compute_Sign_K^l} on each edge of $K^1$ }  \\		                                             
                                                stop $\leftarrow$ stop+1;\\ 
                     \If{ \mbox{\rm sign}$f_{11}$\mbox{\rm sign}$f_{12}\leq 0$ $\wedge$ \mbox{\rm sign}$f_{21}$\mbox{\rm sign}$f_{22}\leq 0$  } {$c \leftarrow$ center of $K^1$;\\
                                             $K_0\leftarrow K^1$;\\
                                               error $\leftarrow \Vert \mbox{F}(c)\Vert$;\\
                                               stop $\leftarrow$ stop-1 ;\\                                                
                                                Pass to next iteration  }                                              
                                            
                                            (sign$f_{11}$,sign$f_{12}$,sign$f_{21}$,sign$f_{22}$) $\leftarrow$ {\it posneg} ( $f_{ij},I_{i2}$,N) \ $i,j=1,2$ ;  {\it $\setminus\setminus$ Here we use Algorithm \ref{Compute_Sign_K^l} on each edge of $K^2$ } \\
                                            \If{ \mbox{\rm sign}$f_{11}$\mbox{\rm sign}$f_{12}\leq 0$ $\wedge$ \mbox{\rm sign}$f_{21}$\mbox{\rm sign}$f_{22}\leq 0$ } {$c \leftarrow$ center of $K^2$;\\
                                             $K_0\leftarrow K^2$;\\
                                               error $\leftarrow \Vert \mbox{F}(c)\Vert$;\\
                                               stop $\leftarrow$ stop-1 ;\\                                                
                                                Pass to next iteration   }                                              
                                             
                                         (sign$f_{11}$,sign$f_{12}$,sign$f_{21}$,sign$f_{22}$) $\leftarrow$ {\it posneg} ( $f_{ij},I_{i3}$,N) \ $i,j=1,2$;  {\it $\setminus\setminus$ Here we use Algorithm \ref{Compute_Sign_K^l} on each edge of $K^3$ }  \\                                             
                                           \If{\mbox{\rm sign}$f_{11}$\mbox{\rm sign}$f_{12}\leq 0$ $\wedge$ \mbox{\rm sign}$f_{21}$\mbox{\rm sign}$f_{22}\leq 0$ } {$c \leftarrow$ center of $K^3$;\\
                                              $K_0\leftarrow K^3$;\\
                                               error $\leftarrow \Vert \mbox{F}(c)\Vert$;\\
                                                stop $\leftarrow$ stop-1 ;\\                                                
                                                Pass to next iteration   }                                              
                                      
                                         (sign$f_{11}$,sign$f_{12}$,sign$f_{21}$,sign$f_{22}$) $\leftarrow$ {\it posneg} ( $f_{ij},I_{i4}$,N) \ $i,j=1,2$;  {\it $\setminus\setminus$ Here we use Algorithm \ref{Compute_Sign_K^l} on each edge of $K^4$ }  \\
                                          \If{ \mbox{\rm sign}$f_{11}$\mbox{\rm sign}$f_{12}\leq 0$ $\wedge$ \mbox{\rm sign}$f_{21}$\mbox{\rm sign}$f_{22}\leq 0$ } {$c \leftarrow$ center of $K^4$;\\
                                              $K_0\leftarrow K^4$;\\
                                               error $\leftarrow \Vert \mbox{F}(c)\Vert$;\\
                                               stop $\leftarrow$ stop-1 ;\\                                                
                                                Pass to next iteration   }                                              
                                     {\it $\setminus\setminus$ Generate the preconditioning system $\mbox{\bf{G}}(X)$ }\\
                                     DFc $\leftarrow$ DF($c$);\\
                                     invDFc $\leftarrow$ inv(DFc);\\    
                                     $f_1 \leftarrow$ invDFc(1,1)*F1orig+invDFc(1,2)*F2orig;
                                     \ \\  
                                     $f_2 \leftarrow$ invDFc(2,1)*F1orig+invDFc(2,2)*F2orig; \\
                                                                                                                                 
                                              }    } {
\Return Wrong $R_0$}
\end{algorithm}
In order to check the accuracy and performance of the algorithm, we test it throughout different systems of equations. We start testing the algorithm in the system given in Example \ref{conv_example}. We took as our starting guess the rectangle $K_0=[0,1]\times [0,1]$, in Table \ref{tab1} we show the behaviour of the sequence throughout different tolerance levels and in Figure \ref{Quadrisection_example1} we illustrate the procedure for tolerance level $10^{-15}$. We have chosen to use 10 digits in the mantissa representation for the root's approximation and its evaluation and 5 digits for the norm evaluation notation. Since the system always verifies the P.M. the algorithm never performs preconditioning.      

\begin{figure}[H]
\begin{minipage}[c]{0.7\textwidth}
\begin{tabular}{| c | @{\vrule height 0pt depth 0pt width 0.1in } l | @{\vrule height 0pt depth 0pt width 0.1in} l | c @{\vrule height 0.14in depth 0in width 0pt}|}\hline    
$\delta$ & $ \ c$ & $\Vert \mbox{\bf{F}}(c)\Vert$ & $\mbox{\rm iter }$\\\hline \hline
$1$                              &   0.500000000  & 0.2788  &        1   \\
                                 &   0.500000000  &         &           \\ \hline
$10^{-1}$                        &   0.062500000  &  0.0586 &        4  \\                             
                                 &   0.937500000  &         &            \\ \hline
 $10^{-2}$                       &   0.007812500  &  0.0077 &        7   \\
                                 &   0.992187500  &         &         \\ \hline
 $10^{-5}$                       &   0.000007629  & 7.6293 1e-06 &  17      \\
                                 &   0.999992370  &                    &        \\ \hline
 $10^{-10}$                      &   5.820799999 1e-11  &  5.8207 1e-11&  34 \\ 
                                 &   0.999999999  &                    &      \\ \hline          
$10^{-15}$                       &   0.000000001 1e-06  &  8.8817 1e-16&  50    \\
                                 &   0.999999999        &              &    \\  \hline                     
\end{tabular}
\captionsetup{width=0.6\linewidth}
\captionof{table}{Evolution of root's approximation, norm evaluation and steps performed throughout different tolerance levels.}\label{tab1}
\end{minipage}
\begin{minipage}[c]{0.1\textwidth}
\includegraphics[height=1.8in,width=1.8in]{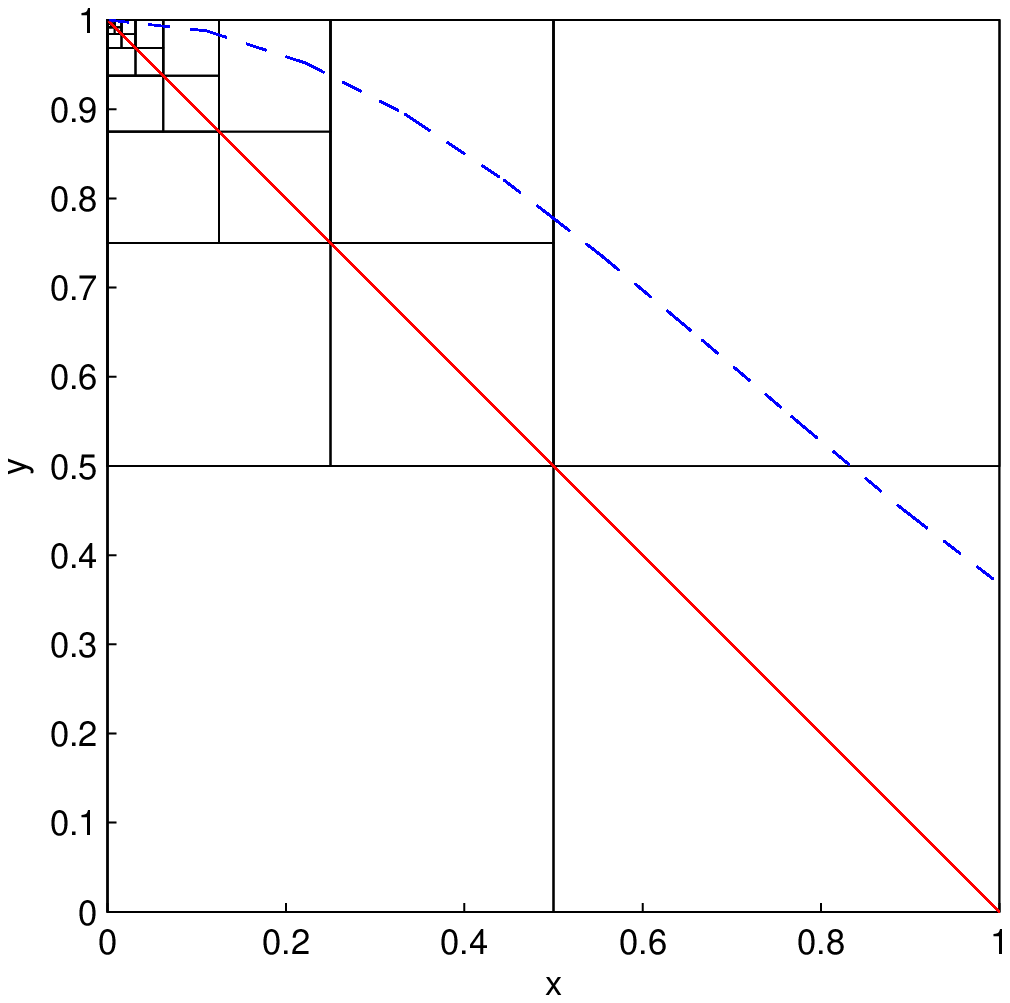} 
\captionsetup{width=1.8in}
\captionof{figure}{Quadrisection procedure.}\label{Quadrisection_example1}
\end{minipage}
\end{figure}

In the following steps we test the algorithm in more difficult problems, we will see that in some systems the algorithm needs to preconditioning in order to guarantee the P.M. conditions throughout the refinement. Let,    
\begin{align*}
&\mbox{\bf{F}}_1(x,y) = (x^2+y^2-1,x-y^2 ) \\ 
&\mbox{\bf{F}}_2(x,y) = (2x-y-e^{-x} , -x+2y-e^{-y}) \\ 
&\mbox{\bf{F}}_3(x,y) = (\sin(x)+\cos(y)+2(x-1) , y-0.5(x-0.5)^2-0.5)\\
&\mbox{\bf{F}}_4(x,y)= (x^2-\cos(xy) , e^{xy}+y )\\
&\mbox{\bf{F}}_5 (x,y)= (x\cos(y)+y\sin(x)-0.5 , e^{e^{-(x+y)}}-y(1+x^2))\\
&\mbox{\bf{F}}_6 (x,y)= (x+5(x-y)^3-1 ,0.5(y-x)^3+y)\\
\end{align*}
be the testing maps. In Table \ref{tab2} we show the numerical performance for the testing maps and in Figure \ref{Quadrisection_example2} we illustrate the algorithm behaviour with the refinement procedure. The systems of equations and their successive possibles preconditioning are represented by a zero contour level on an mesh on the initial guess $K_0$ and the refinement procedure was illustrated using the rectangle \textsc{Matlab}'s functions. The method was implemented setting the tolerance level in $\delta=10^{-15}$ and the interval analysis refinement in N$=3$.    
\begin{table}[H]
\begin{center}
\begin{tabular}{|c|c|c|c|c|}\hline     
$\mbox{\bf{F}}$ & $c$ & $\mbox{\bf{F}}(c)$ & iter & $K_0$ \\\hline \hline

$\mbox{\bf{F}}_1$ & \ 0.618033988749895 &  \  0.004965068306495 1e-14  & 51  & $[0,1] \times [0,1] $                                          \\  
 				  & \ 0.786151377757422 &    -0.123942463016433 1e-14   &   &  \\ \hline

$\mbox{\bf{F}}2$ & \  0.567143290409784 &  \  0.111022302462516 1e-15 &  50 & $[0,1] \times [0,1] $                                          \\  
			      &  \ 0.567143290409784 &  \ 0.111022302462516 1e-15 &    &  \\ \hline

$\mbox{\bf{F}}_3$ & \ 0.378316940137480  & \  0.139577647543639 1e-15  & 51  & $[0,1]\times [0,1]$                        \\
           		  & \ 0.507403383528753  &   -0.072495394968176 1e-15 &	   &   \\  \hline 

$\mbox{\bf{F}}_4$ & \  0.926174872358938 & \  0.129347223584252 1e-15  &  49 & $[0,1]\times [-1,0]$                        \\
          		  & \ -0.582851662173280 &   -0.115653908517277 1e-15 &   & \\ \hline

$\mbox{\bf{F}}_5$ &  \ 0.353246619596717 &   -0.244439451327881 1e-15 & 52 &  $[0,1.1]\times [0,2]$                       \\
             	  &  \ 0.606081736641465 & \  0.047257391058546 1e-15 &    &  \\ \hline 

$\mbox{\bf{F}}_6$ &  \ 0.510030862987151 &   -0.045236309398304 1e-13 & 42 & $[0.4,1]\times [0,0.4]$                       \\
             	  &  \ 0.048996913701194 &   -0.904901681894059 1e-13 &    &   \\ \hline \hline

\end{tabular}
\end{center}
\caption{Root's approximation, evaluation, iteration performed and initial guess for testing maps. }\label{tab2}
\end{table}

\begin{figure}[H]  
\begin{center}
\begin{tabular}{lll}
\includegraphics[height=1.8in,width=1.8in]{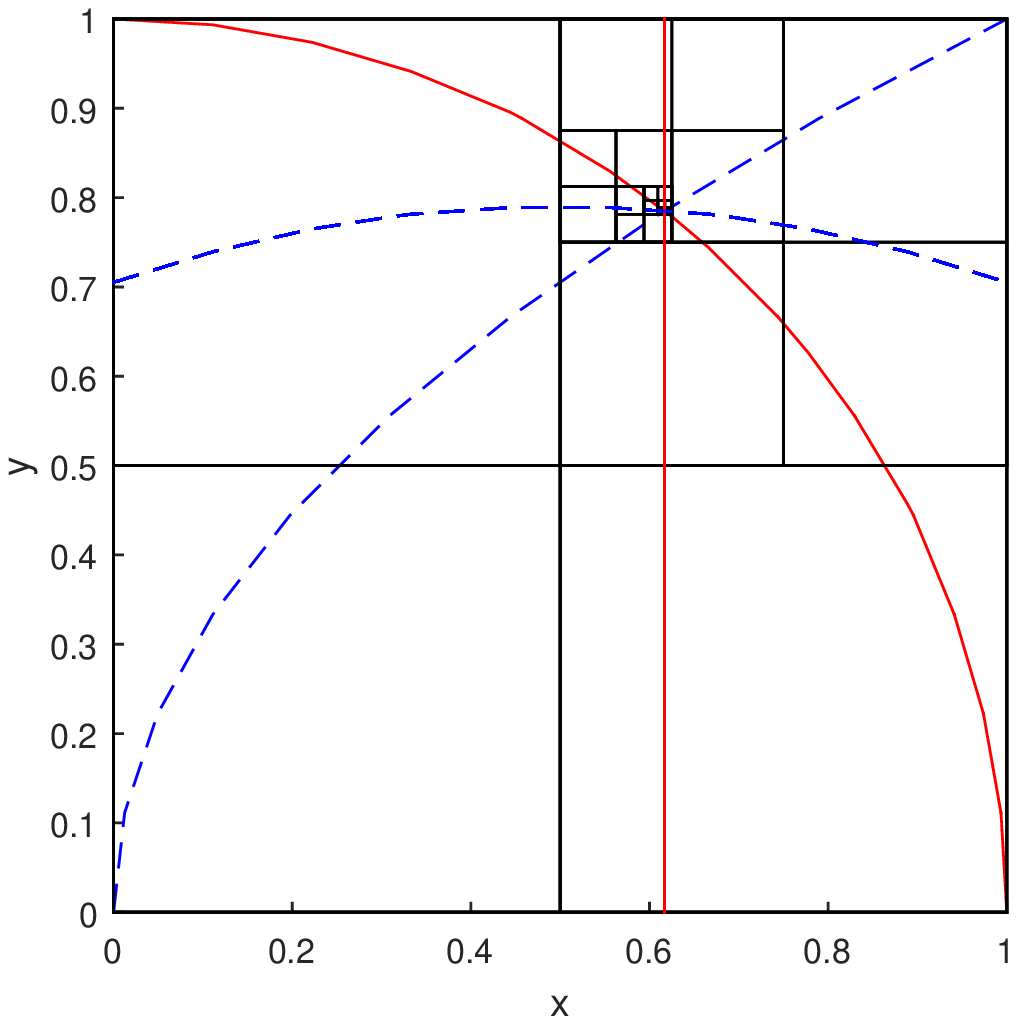} 
\includegraphics[height=1.8in,width=1.8in]{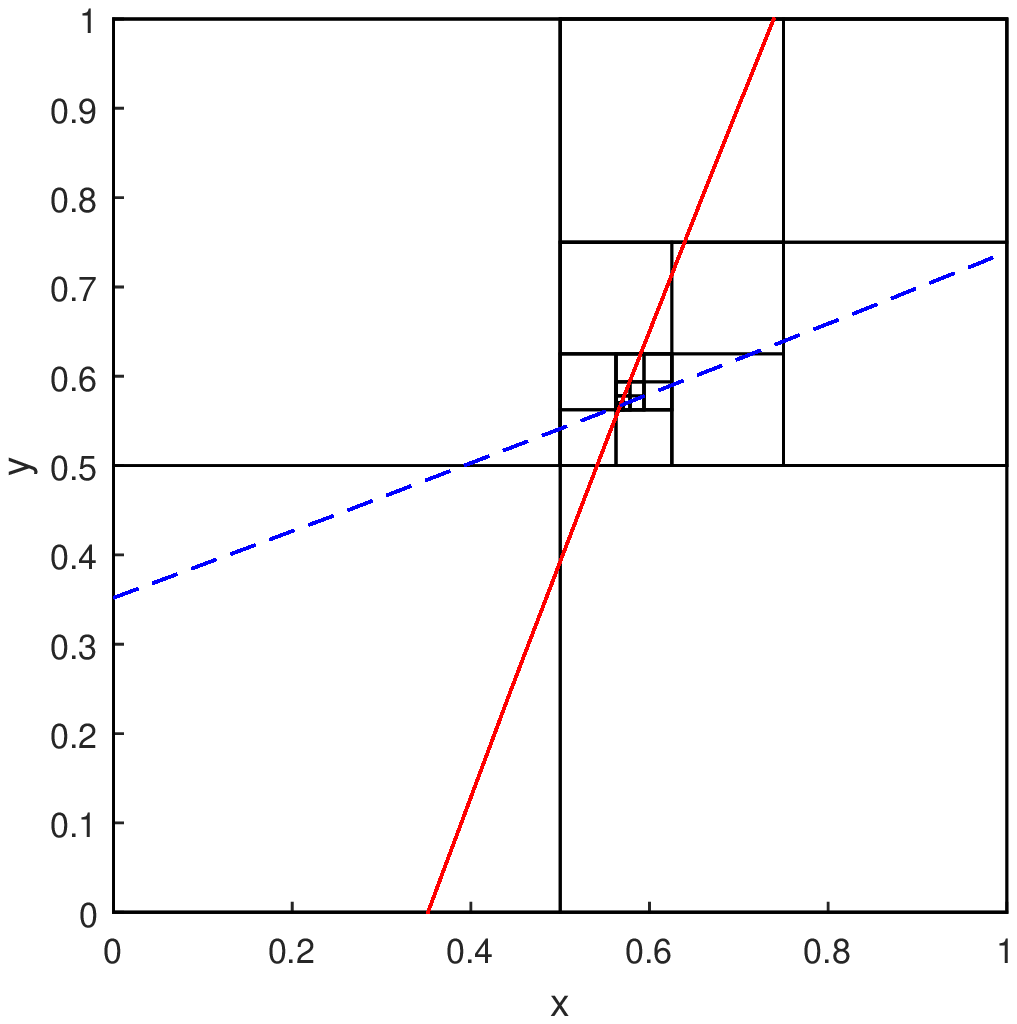} 
\includegraphics[height=1.8in,width=1.8in]{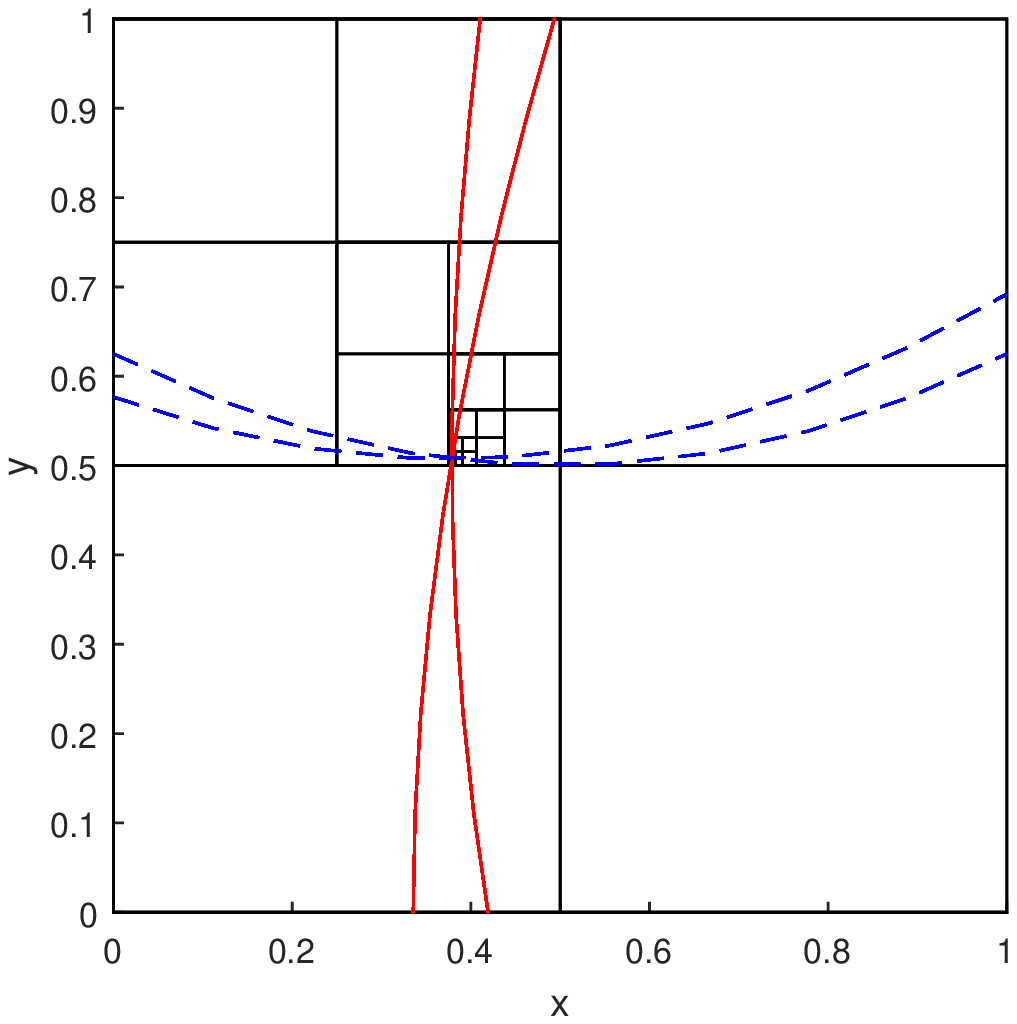}\\ 
\includegraphics[height=1.8in,width=1.8in]{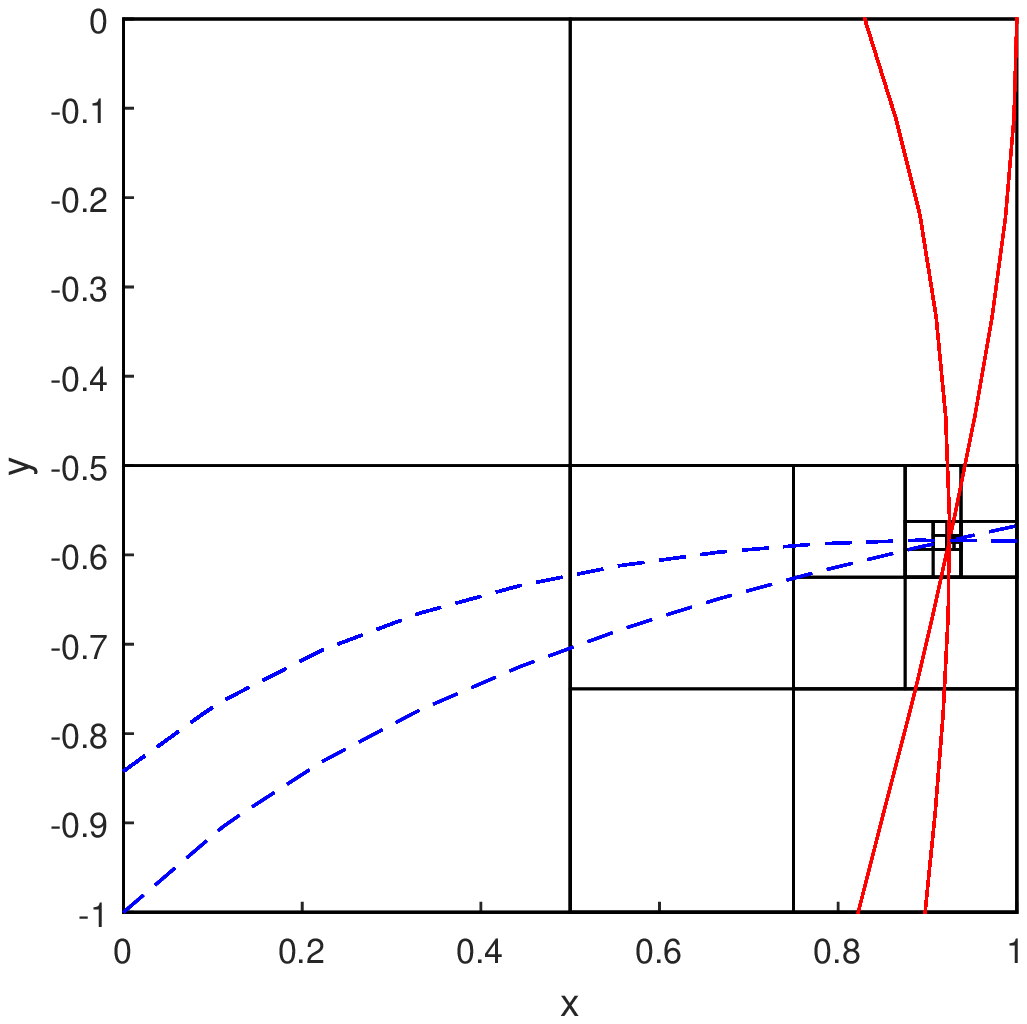}
\includegraphics[height=1.8in,width=1.8in]{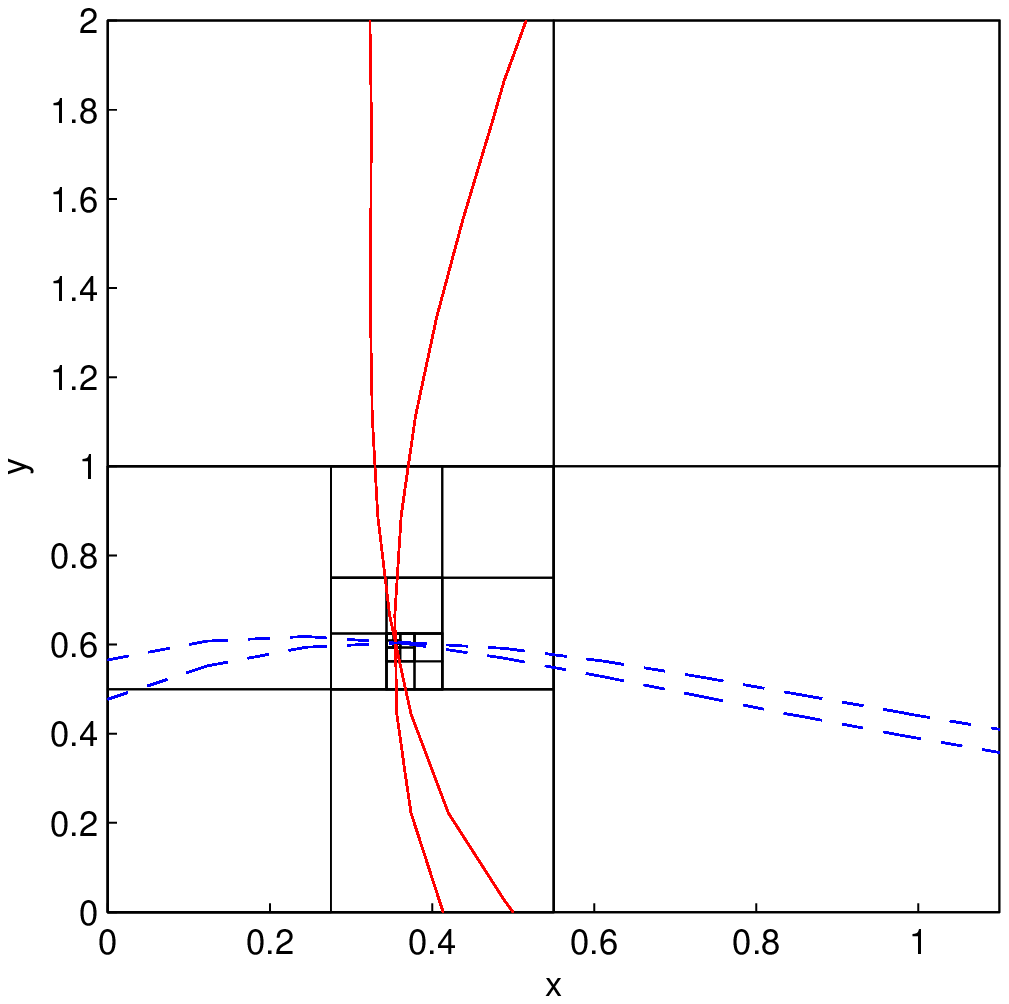}
\includegraphics[height=1.8in,width=1.8in]{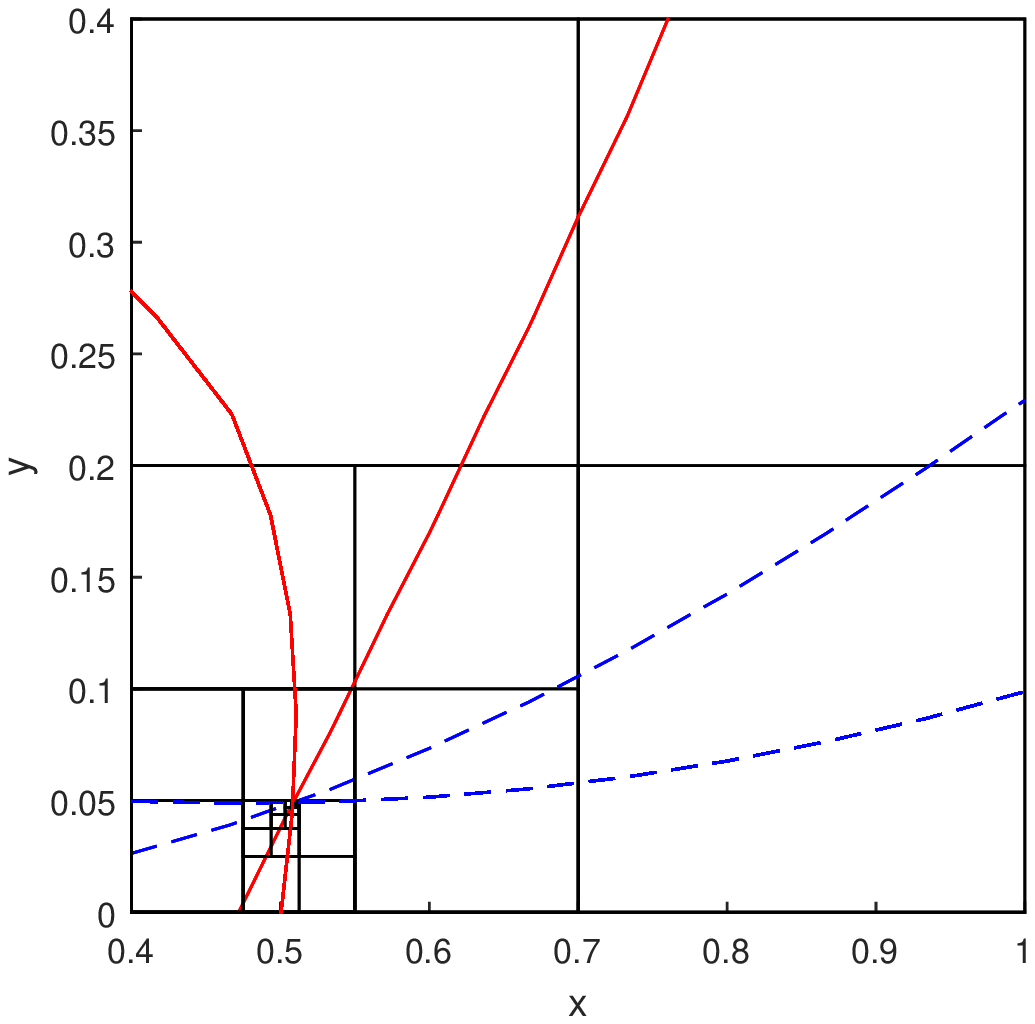}
\end{tabular}

\caption{The first row illustrates the algorithm procedure for $\mbox{\bf{F}}_1,\mbox{\bf{F}}_2, \mbox{\bf{F}}_3   $ and the second for $\mbox{\bf{F}}_4, \mbox{\bf{F}}_5, \mbox{\bf{F}}_6$. The solid red line represents the first coordinate and the dashed blue line represents the second coordinate for the equivalent system $\mbox{\bf{G}}_k(X)=0$.}\label{Quadrisection_example2}
\end{center}
\end{figure}

\section{Conclusion}

In this work we have clarified how a multidimensional bisection algorithm should be performed extending the idea of the classic one dimensional bisection algorithm. Due by the preconditioning in each step we could prove a local convergence theorem and we also found an error estimation. Interval Analysis allowed a fast and reliable way of computing the Poincar\'e-Miranda conditions and the numerical implementation showed that the method has a very good accuracy similar with the classic methods like Newton or continuous optimization.

\end{document}